\documentclass[reqno]{amsart}  

\usepackage{general,resistance}
\renewcommand{\linenopax}{} 

\usepackage[protrusion=true,expansion]{microtype}
 
\usepackage[bookmarks, colorlinks=true, pdfstartview=FitV, linkcolor=blue, citecolor=blue, urlcolor=blue]{hyperref}
\usepackage{cite}  

\usepackage[left=1.4in,right=1.4in,top=1in,bottom=1in]{geometry}

\newcommand{\adj}{\star}

\newcommand{\LapK}{{\ensuremath{\Lap_{\HE}^{(Kr)}}}\xspace}     
\newcommand{\dx}{\operatorname{d}\negsp[4]x}  		
\newcommand{\dy}{\operatorname{d}\negsp[4]y}  		
\newcommand{\dgl}{\operatorname{d}\negsp[4]\gl}  		
\newcommand{\dgm}{\operatorname{d}\negsp[4]\gm}  		
\renewcommand{\cl}[1]{\cj{#1}}            

\newcommand{\sgn}{\operatorname{sgn}\xspace}

\numberwithin{equation}{section} \numberwithin{theorem}{section}
\numberwithin{figure}{section}

\begin{document}


\title[Unbounded containment and Krein Laplacian]
  {Unbounded containment in the energy space of a network and the Krein extension of the energy Laplacian}

\author{Palle E. T. Jorgensen}\email{palle-jorgensen@uiowa.edu} 
\address{University of Iowa, Iowa City, IA 52246-1419 USA}

\author{Erin P. J. Pearse}\email{epearse@calpoly.edu}
\address{California State Polytechnic University, San Luis Obispo, CA 93405-0403 USA}


\begin{abstract}
  We compare the space of square-summable functions on an infinite graph (denoted $\ell^2(G)$) with the space of functions of finite energy (denoted $\mathcal{H}_{\mathcal{E}}$). There is a notion of inclusion that allows $\ell^2(G)$ to be embedded into $\mathcal{H}_{\mathcal{E}}$, but the required inclusion operator is unbounded in most interesting cases. These observations assist in the construction of the Krein extension of the Laplace operator on $\mathcal{H}_{\mathcal{E}}$. We investigate the Krein extension and compare it to the Friedrichs extension developed by the authors in a previous paper.
\end{abstract}




\maketitle

\allowdisplaybreaks


\section{Introduction}

We study Laplace operators on infinite networks, and their self-adjoint extensions. Here, a network is just an connected undirected weighted graph $(G, c)$; see Definition~\ref{def:network}. The associated network Laplacian \Lap acts on functions $u:G \to \bR$; see Definition~\ref{def:graph-laplacian}. We study the case when \Lap is unbounded, in which case some care must be taken with the domains. One natural domain for \Lap lies in \HE, the Hilbert space of (equivalence classes of) finite-energy functions; see Definition~\ref{def:H_energy}. Another natural domain for \Lap lies in $\ell^2(G)$, the unweighted space of square-summable functions on $G$ under counting measure; see Definition~\ref{def:ell2}. We will use the respective notations \LapE and $\Lap_2$ to refer to these two very distinct incarnations of the Laplacian. Although the action of \Lap is defined by the same formula for elements of $\ell^2(G)$ as for \HE (i.e., the operators \LapE and $\Lap_2$ agree formally), the difference in domains results in rather striking spectral theoretic consequences in the two different contexts. 

Common to \LapE and $\Lap_2$ is that each is defined on its natural dense domain in each of the Hilbert spaces \HE and $\ell^2(G)$, and in each case it is a Hermitian and non-negative operator. However, it is known from \cite{Woj07, SRAMO, KellerLenz09, KellerLenz10} that \Lap is essentially self-adjoint on its natural domain in $\ell^2(G)$ but in \cite{SRAMO} it is shown that \Lap is \emph{not} essentially self-adjoint on its natural domain in \HE (see Definition~\ref{def:domLapE}) (this fact is used in the proof of Theorem~\ref{def:M=Lstar}, and elsewhere). When \Lap is not essentially self-adjoint, it has multiple self-adjoint extensions. The Friedrichs extension was constructed in a previous paper \cite{Friedrichs}; in the present paper, we construct the Krein extension (see Theorem~\ref{thm:JJ-Lap}) and compare it to the Friedrichs extension (see Cor.~\ref{thm:spec(LapK)}). To carry this out, we investigate the inclusion operator $J$ that maps finitely supported functions from $\Lap_2$ into \LapE. In particular, we use $J$ and its adjoint to construct the Krein extension (Theorem~\ref{thm:JJ-Lap}), find bounds for \Lap (Theorem~\ref{thm:equivalent-bounds}), and study \Lap as a mapping from \HE to $\ell^2(G)$ (which, it turns out, is the adjoint of $J$; cf.~Theorem~\ref{def:M=Lstar}). For background on the Krein extensions, see, for example, \cite{MR2349811,MR1618628,MR595414}.

To make this paper accessible to diverse audiences, we have included a number of definitions we shall need from the theory of (i) infinite networks, and (ii) the use of unbounded operators on Hilbert space in discrete contexts. Some useful background references for the first are \cite{Soardi94} and \cite{Woess09} and the multifarious references cited therein; see also \cite{Yamasaki79, Zem91, HK10, Kayano88, Kayano84, Kayano82, MuYaYo, vBL09, DuJo10}. For the second, see \cite{DuSc88} (especially Ch.~12) and \cite{vN32, Sto90, DJ06, BB09}.
For relevant background on reproducing kernels, see e.g., \cite{PaSc72, Aronszajn50, MuYaYo, Kal70}. 
   In our first section below, we have recorded some lemmas from \cite{DGG, ERM, Multipliers, SRAMO, bdG, RBIN, RANR, Interpolation, LPS, OTERN} in the form in which they will be needed in the rest of the paper. Some of these results are folkloric or well known in the literature; in such cases, we refer to our own papers only for convenience.
   
Some of our results concern an operator which may at first appear to be too banal to be of interest: the inclusion operator mapping a dense subspace of one Hilbert space into another Hilbert space. It will be important in what follows to see that the inclusion is closable, or equivalently, that the adjoint of the inclusion is densely defined. To indicate how this can fail (rather severely), we include the following simple example. For further background on closability, see \cite{MR0052042,MR600620}.

\begin{exm}\label{exm:trivial-adjoint}
	Let $X=[0,1]$, and consider $L^2(X,\gl)$ and $L^2(X,\gm)$ for measures \gl and \gm which are mutually singular. For concreteness, let \gl be Lebesgue measure, and let \gm be the classical singular continuous Cantor measure. Then the support of \gm is the middle-thirds Cantor set, which we denote by $K$, so that $\gl(K)=1$ and $\gl(X\less K)=1$. The continuous functions $C(X)$ are a dense subspace of both $L^2(X,\gl)$ and $L^2(X,\gm)$ (see, e.g. \cite[Ch.~2]{Rud87}). Define the inclusion operator $J$ to be the operator with dense domain $C(X)$ and
\linenopax
\begin{align}\label{eqn:L2-inclusion-exm}
	J:C(X) \ci L^2(X,\gl) \to L^2(X,\gm)
	\qq\text{by}\qq
	J \gf = \gf.
\end{align}
We will show that $\dom J^\ad = \{0\}$, so suppose $f \in \dom J^\ad$. Without loss of generality, one can assume $f \geq 0$ by replacing $f$ with $|f|$, if necessary.
By definition, $f \in \dom J^\ad$ iff there exists $g \in L^2(X,\gl)$ for which
\linenopax
\begin{align}\label{eqn:adj-exm-defn}
	\la J\gf,f\ra_{\gm} = \int_X \gf f \dgm = \int_X \gf g \dgl = \la \gf, g \ra_{\gl},
	\qq \text{for all }\gf \in C(X). 
\end{align}
One can choose $(\gf_n)_{n =1}^\iy \ci C(X)$ so that $\gf_n|_{K}=1$ and $\lim_{n \to \iy} \int_X \gf_n \dgl = 0$ by considering the appropriate piecewise linear modifications of the constant function 1. For example, see Figure~\ref{fig:cantor-collapse}.
	\begin{figure}[b]
		\begin{centering}
		\scalebox{0.7}{\includegraphics{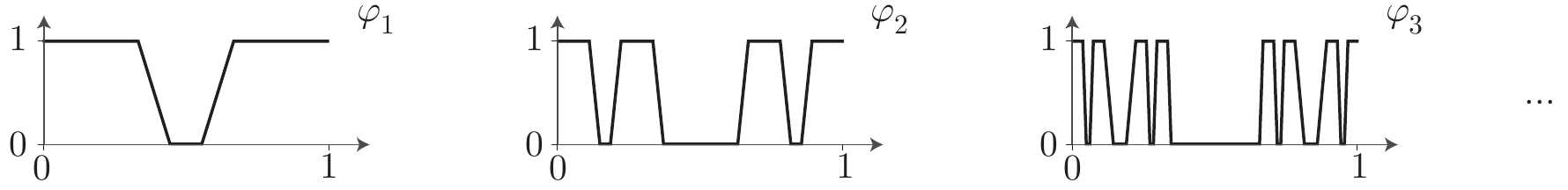}} 
		\end{centering}
		\caption{\captionsize A sequence $\{\gf_n\} \ci C(X)$ for which $\gf_n|_K=1$ and $\lim \int_X \gf_n \dgl =0$. See Example~\ref{exm:trivial-adjoint}.}
		\label{fig:cantor-collapse}
	\end{figure}
Now we have 
\linenopax
\begin{align}\label{eqn:adj-exm-defn}
	\int_X |f| \dgm 
	= \la 1, f \ra_{\gm} 
	= \la \gf_n, f \ra_{\gm} 
	= \la \gf_n, J^\ad f \ra_{\gl},
	\qq\text{for any } n,
\end{align}
but $\lim_{n \to \iy} \int_X \gf_n g \dgl = 0$ for any continuous $g \in L^2(X,\gl)$, which shows that $f=0$. So $\dom J^\ad = \{0\}$. In this context, one can interpret the adjoint of the inclusion as multiplication by a Radon-Nikodym derivative (``$J^\ad f  = f\frac{\dgm}{\dgl}$''), which must be trivial when the measures are mutually singular. Hence, the inclusion operator in \eqref{eqn:L2-inclusion-exm} is not closable.
\end{exm}

\subsection{Acknowledgements}

The authors are grateful to Daniel Lenz for fruitful discussions and generosity of ideas, especially regarding the inclusion map $J:\ell^2(G) \to \HE$ (and its adjoint) which appears in Definition~\ref{def:J} and is discussed throughout \S\ref{sec:inclusion-and-Krein}. In particular, we are grateful for contributions to the statement and proof of Theorem~\ref{thm:JJ-Lap}, a key result for establishing the identity of $J^\ad J$ as the $\ell^2$ Laplacian, and $JJ^\ad$ as the Krein extension of the Laplacian on \HE. See also \cite{KellerLenz09,KellerLenz10,HaeselerKellerLenzWoj12}.

\section{Basic terms and previous results}
\label{sec:Basic-terms-and-previous-results}

We now proceed to introduce the key notions used throughout this paper: resistance networks, the energy form \energy, the Laplace operator \Lap, and their elementary properties. 

\begin{defn}\label{def:network}
  A \emph{(resistance) network} $(\Graph,\cond)$ is a connected weighted undirected graph with vertex set \Graph and adjacency relation defined by a symmetric \emph{conductance function} $\cond: \Graph \times \Graph \to [0,\iy)$. 
  More precisely, there is an edge connecting $x$ and $y$ iff $c_{xy}>0$, in which case we write $x \nbr y$. The nonnegative number $c_{xy}=c_{yx}$ is the weight associated to this edge and it is interpreted as the conductance, or reciprocal resistance of the edge.
  
  We make the standing assumption that $(\Graph,\cond)$ is \emph{locally finite}. This means that every vertex has \emph{finite degree}, i.e., for any fixed $x \in \Graph$ there are only finitely many $y \in \Graph$ for which $c_{xy}>0$. We denote the net conductance at a vertex by 
  \linenopax
  \begin{align}\label{eqn:c(x)}
      \cond(x) := \sum_{y \nbr x} \cond_{xy}.     
  \end{align}
Motivated by current flow in electrical networks, we also assume $c_{xx}=0$ for every vertex $x \in G$.  

In this paper, \emph{connected} means simply that for any $x,y \in \Graph $, there is a finite sequence $\{x_i\}_{i=0}^n$ with $x=x_0$, $y=x_n$, and $\cond_{x_{i-1} x_i} > 0$, $i=1,\dots,n$.  

For any network, one can fix a reference vertex, which we shall denote by $o$ (for ``origin''). It will always be apparent that our calculations depend in no way on the choice of $o$.
\end{defn}

\begin{defn}\label{def:graph-laplacian}
  The \emph{Laplacian} on \Graph is the linear difference operator which acts on a function $u:\Graph \to \bR$ by
  \linenopax
  \begin{equation}\label{eqn:Lap}
    (\Lap u)(x) :
    = \sum_{y \nbr x} \cond_{xy}(u(x)-u(y)).
  \end{equation}
  A function $u:\Graph \to \bR$ is \emph{harmonic} iff $\Lap u(x)=0$ for each $x \in \Graph$.
  Note that the sum in \eqref{eqn:Lap} is finite by the local finiteness assumption above, and so the Laplacian is well-defined.
\end{defn}

\begin{remark}\label{rem:Laplace-incarnations}
	The domain of \Lap, considered as an operator on \HE or $\ell^2(G)$, is given in Definition~\ref{def:domLapE} and Definition~\ref{def:domLap2}.
\end{remark}

\subsection{The energy space \HE} 
\label{sec:The-energy-space}

\begin{defn}\label{def:graph-energy}
  The \emph{energy form} is the (closed, bilinear) Dirichlet form
  \linenopax
  \begin{align}\label{eqn:def:energy-form}
    \energy(u,v)
    := \frac12 \sum_{x,y \in \Graph} \cond_{xy}(u(x)-u(y))(v(x)-v(y)),
  \end{align}
  which is defined whenever the functions $u$ and $v$ lie in the domain
  \linenopax
  \begin{equation}\label{eqn:def:energy-domain}
    \dom \energy = \{u:\Graph \to \bR \suth \energy(u,u)<\iy\}.
  \end{equation}
  Hereafter, we write the energy of $u$ as $\energy(u) := \energy(u,u)$. Note that $\energy(u)$ is a sum of nonnegative terms and hence converges iff it converges absolutely. 
\end{defn}

The energy form \energy is sesquilinear and conjugate symmetric on $\dom \energy$ and would be an inner product if it were positive definite. Let \one denote the constant function with value 1 and observe that $\ker \energy = \bR \one$. One can show that $\dom \energy / \bR \one$ is complete and that \energy is closed;  see \cite{DGG,OTERN}, \cite{Kat95}, or \cite{FOT94}.

\begin{defn}\label{def:H_energy}\label{def:The-energy-Hilbert-space}
  The \emph{energy (Hilbert) space} is $\HE := \dom \energy / \bR \one$. The inner product and corresponding norm are denoted by
  \linenopax
  \begin{equation}\label{eqn:energy-inner-product}
    \la u, v \ra_\energy := \energy(u,v)
    \q\text{and}\q
    \|u\|_\energy := \energy(u,u)^{1/2}.
  \end{equation}
\end{defn}

It is shown in \cite[Lem.~2.5]{DGG} that the evaluation functionals $L_x u = u(x) - u(o)$ are bounded, and hence correspond to elements of \HE by Riesz duality (see also \cite[Cor.~2.6]{DGG}). When considering \bC-valued functions, \eqref{eqn:energy-inner-product} is modified as follows: $\la u, v \ra_\energy := \energy(\cj{u},v)$.

\begin{defn}\label{def:vx}\label{def:energy-kernel}
  Let $v_x$ be defined to be the unique element of \HE for which
  \linenopax
  \begin{equation}\label{eqn:v_x}
    \la v_x, u\ra_\energy = u(x)-u(o),
    \qq \text{for every } u \in \HE.
  \end{equation}
  Note that $v_o$ corresponds to a constant function, since $\la v_o, u\ra_\energy = 0$ for every $u \in \HE$. Therefore, $v_o$ may be safely omitted in some calculations. 
\end{defn}

  As \eqref{eqn:v_x} means that the collection $\{v_x\}_{x \in \Graph}$ forms a reproducing kernel for \HE, we call $\{v_x\}_{x \in \Graph}$ the \emph{energy kernel}. It follows that the energy kernel has dense span in \HE; cf. \cite{Aronszajn50}.\footnote{To see this, note that a RKHS is a Hilbert space $H$ of functions on some set $X$, such that point evaluation by points in $X$ is continuous in the norm of $H$. Consequently, every $x \in X$ defines a vector $k_x \in H$ by Riesz's Theorem, and it is immediate from this that $\spn\{k_x\}_{x \in X}$ is dense in $H$.}

\begin{remark}[Differences and representatives]\label{rem:differences}
  Equation \eqref{eqn:v_x} is independent of the choice of representative of $u$ because the right-hand side is a difference: if $u$ and $u'$ are both representatives of the same element of \HE, then $u' = u+k$ for some $k \in \bR$ and
  $u'(x) - u'(o) = (u(x)+k)-(u(o)+k) = u(x)-u(o).$
  By the same token, the formula for \Lap given in \eqref{eqn:Lap} describes unambiguously the action of \Lap on equivalence classes $u \in \HE$. Indeed, formula \eqref{eqn:Lap} defines a function $\Lap u:\Graph \to \bR$ but we may also interpret $\Lap u$ as the class containing this representative.
\end{remark}

\begin{defn}\label{def:d_x}
  Let $\gd_x \in \ell^2(G)$ denote the Dirac mass at $x$, i.e., the characteristic function of the singleton $\{x\}$ and let $\gd_x \in \HE$ denote the element of \HE which has $\gd_x \in \ell^2(G)$ as a representative. The context will make it clear which meaning is intended. 
\end{defn}

\begin{remark}\label{rem:Diracs-in-HE}
	 Observe that $\energy(\gd_x) = \cond(x) < \iy$ is immediate from \eqref{eqn:def:energy-form}, and hence one always has $\gd_x \in \HE$ (recall that $c(x)$ is the total conductance at $x$; see \eqref{eqn:c(x)}).
\end{remark}

\begin{defn}\label{def:Fin}
  For $v \in \HE$, one says that $v$ has \emph{finite support} iff there is a finite set $F \ci G$ such that $v(x) = k \in \bC$ for all $x \notin F$. Equivalently, the set of functions of finite support in \HE is 
  \linenopax
  \begin{equation}\label{eqn:span(dx)}
    \spn\{\gd_x\} = \{u \in \dom \energy \suth u \text{ is constant outside some finite set}\}.
  \end{equation}
  Define \Fin to be the \energy-closure of $\spn\{\gd_x\}$. 
\end{defn}

\begin{defn}\label{def:Harm}
  The set of harmonic functions of finite energy is denoted
  \linenopax
  \begin{equation}\label{eqn:Harm}
    \Harm := \{v \in \HE \suth \Lap v(x) = 0, \text{ for all } x \in G\}.
  \end{equation}
\end{defn}

The following result is well known; see \cite[\S{VI}]{Soardi94}, \cite[\S9.3]{Lyons}, \cite[Thm.~2.15]{DGG}, or the original \cite[Thm.~4.1]{Yamasaki79}.

\begin{theorem}[Royden Decomposition]\label{thm:HE=Fin+Harm}
  $\HE = \Fin \oplus \Harm$.
\end{theorem}
 

\begin{defn}\label{def:dipole}
  A \emph{monopole} is any $w \in \HE$ satisfying the pointwise identity $\Lap w = \gd_x$ (in either sense of Remark~\ref{rem:differences}) for some vertex $x \in \Graph$. 
  A \emph{dipole} is any $v \in \HE$ satisfying the pointwise identity $\Lap v = \gd_x - \gd_y$ for some $x,y \in \Graph$.%
\end{defn}

\begin{remark}\label{rem:monotransience}
  It is easy to see from the definitions (or \cite[Lemma~2.13]{DGG}) that energy kernel elements are dipoles, i.e., that $\Lap v_x = \gd_x - \gd_o$, and that one can therefore always find a dipole for any given pair of vertices $x,y \in G$, namely, $v_x-v_y$. On the other hand, monopoles exist if and only if the network is transient (see \cite[Thm.~2.12]{Woess00} or \cite[Rem.~3.5]{DGG}). 
\end{remark}  

\begin{remark}\label{rem:normalization}
  Denote the unique energy-minimizing monopole at $o$ by $w_o$; the existence of such an object is explained in \cite[\S3.1]{DGG}.
  We will be interested in the family of monopoles defined by
  \linenopax
  \begin{align}\label{eqn:w_x}
    \monov := w_o + v_x, \qq x \neq o.
  \end{align}
  We will use the representatives specified by
  \linenopax
  \begin{align}\label{eqn:w_x(o)}
    \monov(y) = \la \monov,\monoy\ra_\energy = \monoy(x), 
    \qq\text{and}\qq
    v_x(o)=0.
  \end{align}
  When $\Harm=0$, $\energy(\monov) = \la \monov,\monov\ra_\energy = \monov(x)$ is the \emph{capacity} of $x$; see, e.g., \cite[\S4.D]{Woess09}.
\end{remark}

\begin{lemma}[{$\negsp[5]$\cite[Lem.~2.11]{DGG}}]
  \label{thm:<delta_x,v>=Lapv(x)}
  For $x \in \Graph$ and $u \in \HE$,  
  \linenopax
  \begin{align}\label{eqn:dx-rep}
		\la \gd_x, u \ra_\energy = \Lap u(x).
	\end{align}
  \begin{proof}
    One can compute $\la \gd_x, u \ra_\energy = \energy(\gd_x, u)$ directly from formula \eqref{eqn:def:energy-form}.
  \end{proof}
\end{lemma}

\begin{lemma}[{$\negsp[5]$\cite[Lem.~2.19]{DGG}}]
  \label{thm:d=v}
	Whenever $\deg(x)<\iy$, one can express $\gd_x$ in terms of the reproducing kernel $\{v_x\}_{x \in G}$ via
  \linenopax
	\begin{align}\label{eqn:d=v}
		\gd_x = c(x)v_x - \sum_{y \nbr x} c_{xy}v_y.
	\end{align}
\end{lemma}

\begin{lemma}\label{thm:Lap-mono-Kron}
  For any $x,y \in G$, 
  \linenopax
  \begin{align}\label{eqn:Lap-mono-Kron}
    \Lap \monov(y) = \Lap \monoy(x) 
    = \la \monov,\Lap \monoy\ra_\energy 
    = \la \Lap \monov, \monoy\ra_\energy = \gd_{xy},  
  \end{align}
  where $\gd_{xy}$ is the Kronecker delta.
  \begin{proof}
    First, note that $\Lap \monov(y) = \gd_{xy} = \Lap \monov(y)$ as functions, immediately from the definition of monopole. Then the substitution $\Lap \monoy = \gd_y$ gives 
    \linenopax
  	\begin{align}\label{eqn:<w,d>=d}
      \la \monov,\Lap \monoy\ra_\energy
      = \la \monov,\gd_y\ra_\energy
      = \Lap \monov(y)
    \end{align}
    by Lemma~\ref{thm:<delta_x,v>=Lapv(x)}, and similarly for the other identity. 
  \end{proof}
\end{lemma}

\begin{defn}\label{def:domLapE}
  The closed operator \LapE on \HE is obtained by taking the graph closure of the operator \Lap defined on $\spn\{\monov\}_{x \in G}$ pointwise by \eqref{eqn:Lap}. 
  %
\end{defn}

The following lemma shows that Definition~\ref{def:domLapE} makes sense.  

 
\begin{lemma}\label{thm:semibounded}
   \LapE is a well-defined, non-negative, closed and Hermitian operator on \HE. 
  \begin{proof}
    Let $\gx = \sum_{x \in F} \gx_x \monov$, for some finite set $F \ci G$. By \eqref{eqn:Lap-mono-Kron}, 
    \linenopax
    \begin{align}\label{eqn:semibounded}
      \la u, \Lap u \ra_\energy
      &= \sum_{x,y \in F} {\gx_x} \gx_y \la \monov,\Lap \monoy\ra_\energy 
       = \sum_{x,y \in F} {\gx_x} \gx_y \gd_{xy}
       = \sum_{x \in F} |\gx_x|^2 \geq 0.  
    \end{align} 
    Since the conductance function $c$ is \bR-valued, the Laplacian commutes with conjugation and therefore is also symmetric as an operator in the corresponding \bC-valued Hilbert space. This implies \Lap is Hermitian and hence contained in its adjoint. Since every adjoint operator is closed, \Lap is closable. Furthermore, the closure of any semibounded operator is semibounded.    
    To see that the image of \Lap lies in \HE, note from Lemma~\ref{thm:Lap-mono-Kron} that $\Lap\monov = \gd_x \in \HE$ by Remark~\ref{rem:Diracs-in-HE}.
  \end{proof}
\end{lemma}


\subsection{The Hilbert space $\ell^2(G)$} 
\label{sec:ell2(G)}

As there are many uses of the notation $\ell^2(G)$, we provide the following elementary definitions to clarify our conventions.

\begin{defn}\label{def:ell2}
For functions $u,v:G \to \bR$, define the inner product
  \linenopax
  \begin{align}\label{eqn:ell2-inner-product}
    \la u, v\ra_2 := \sum_{x \in G} u(x) v(x).
  \end{align}
\end{defn}
  
\begin{defn}\label{def:domLap2}
  The closed operator $\Lap_2$ on $\ell^2(G)$ is obtained by taking the graph closure (see Remark~\ref{rem:domLap2}) of the operator \Lap which is defined pointwise by \eqref{eqn:Lap} on $\spn\{\gd_x\}_{x \in G}$, the subspace of (finite) linear combinations of point masses.
\end{defn}

\begin{remark}\label{rem:domLap2}
  \cite[Lem.~2.7 and Thm.~2.8]{SRAMO} states that $\Lap$ is semibounded and essentially self-adjoint as an operator on $\spn\{\gd_x\}_{x \in G}$. We provide a new proof of essentially self-adjointness that holds more generally (and has a much simpler proof); see Theorem~\ref{thm:no-defect} and Corollary~\ref{thm:ess-sa} just below.
  
  Since $\Lap$ is semibounded and essentially self-adjoint, it follows that $\Lap$ is closable by the same arguments as in the end of the proof of Lemma~\ref{thm:semibounded}, whence $\Lap_2$ is closed, self-adjoint, and in particular, well-defined. Note that in sharp contrast, the analogous operator \LapE is not automatically self-adjoint (see \cite{SRAMO}) and hence some care is needed (for example, in the proof of Lemma~\ref{thm:semibounded}).
  See also \cite{Woj07, KellerLenz09, KellerLenz10}.
\end{remark}

\begin{defn}\label{def:C_0(G)}
	Denote the space of functions on $G$ which vanish at \iy by $C_0(G)$.
	Note that $C_0(G)$ contains all compactly supported functions and $\ell^p(G)$, for any $1 \leq p < \iy$. Additionally, note that if $f \in C_0(G)$ and $f(x)>0$ for some $x \in G$, then $f$ or $-f$ has a local maximum in $G$.
\end{defn}

\begin{theorem}\label{thm:no-harm-in-l2}
	There is no nonzero $u \in C_0(G)$ which satisfies the pointwise equation $\Lap u(x) = -u(x)$ at every $x \in G$. In particular, if $u \in \ell^2(G)$ and $\Lap u=0$, then $u=0$.
\end{theorem}

\begin{remark}\label{rem:no-harm-in-l2}
	The proof of Theorem~\ref{thm:no-harm-in-l2} proceeds by selecting a nonzero function $u$ and observing that (since it is chosen from $C_0(G)$, $\ell^2(G)$, or some similar set) that a nonzero function in the given set must have a local maximum; this is then used to show that $\Lap u=0$ iff $u$ is constant. This result is well known and a proof may be found in \cite{KemenySnellKnapp}, \cite{Woess09,Woess00}, or \cite{LevPerWil08} (the last source requires a slight adaptation for infinite networks). 
	
	It is less well-known, however, that this same strategy can be adapted to show \Lap is essentially self-adjoint on $\ell^2(G)$, and this is the content of the following theorem from the Ph.D. thesis of Radek Wojciechowski. We include the proof here for completeness and to adapt the notation to the present setting.
\end{remark}

\begin{theorem}[{$\negsp[5]$ \cite[Thm.~1.3.1]{Woj07}}]\label{thm:no-defect}
	There is no nonzero $u \in C_0(G)$ which satisfies the pointwise equation $\Lap u(x) = -u(x)$ at every $x \in G$.
\end{theorem}
\begin{proof}
	From \eqref{eqn:Lap} and \eqref{eqn:c(x)}, such a solution \gf must satisfy
	\linenopax
	\begin{align}
		0 = \Lap \gf (x) + \gf(x) 
			&= \left(1 + c(x)\right)\gf(x) - c(x) \sum_{y \nbr x} p(x,y) \gf(y),
	\end{align}
	for $p(x,y) = \frac{c_{xy}}{c(x)}$. Thus we have
	\linenopax
	\begin{align}\label{eqn:P-defect-identity}
		(P\gf)(x) = \left(1 + \frac1{c(x)}\right)\gf(x),
	\end{align}
	where $(P\gf)(x) = \sum_{y \nbr x} p(x,y) \gf(y)$. If $\gf \in C_0(G)$ and \gf is nonzero, there must exist a vertex $x_1$ which is a local maximizer or minimizer of \gf, and we can assume without loss of generality that $\gf(x_1) > 0$ (see Definition~\ref{def:C_0(G)}). Since $x_1$ is a maximizer of \gf,
	\linenopax
	\begin{align*}
		\gf(x_1)  \geq  (P\gf)(x_1)  =  \left(1 + \frac1{c(x_1)}\right)\gf(x_1),
	\end{align*}
	which implies $0 \geq \frac1{c(x_1)} \gf(x_1)$. Since $c(x)>0$ for all $x \in G$,  this contradicts $\gf(x_1) > 0$.
\end{proof}

\begin{cor}\label{thm:ess-sa}
	$\Lap_{\ell^2}$ is essentially self-adjoint.
\end{cor}
\begin{proof}
	We show that $\Lap_{\ell^2}$ has no defect vector, using Theorem~\ref{thm:no-defect}. First, note that $f$ is a defect vector of $\Lap_{\ell^2}$ iff
	\linenopax
	\begin{align}\label{eqn:}
		&\la \gf + \Lap_{\ell^2} \gf, f\ra_{\ell^2} = 0, \qq \text{for all } \gf \in \spn\{\gd_x\}_{x \in G} \\
		&\iff \la \gd_x + \Lap_{\ell^2} \gd_x, f\ra_{\ell^2} = 0, \qq \text{for all } x \in G \\
		&\iff \left\la \gd_x + c(x) \gd_x - \sum_{y \nbr x} c_{xy} \gd_y, f \right\ra_{\ell^2} = 0, \qq \text{for all } x \in G \\
		&\iff f(x) + c(x) f(x) - \sum_{y \nbr x} c_{xy} f(y) = 0, \qq \text{for all } x \in G.
	\end{align}
	Then by Theorem~\ref{thm:no-defect}, $f=0$.
\end{proof}

\section{The inclusion operator and Krein extension}
\label{sec:inclusion-and-Krein}


\subsection{The inclusion operator}

We consider Dirac masses $\gd_x$ as elements of $\ell^2(G)$ and also as elements of \HE, and consequently, we can also consider $\spn\{\gd_x\}_{x \in G}$ (the set of \emph{finite} linear combinations of elements of the form $\gd_x$) as a subspace of $\ell^2(G)$ and also as a subspace of \HE. (Note that while $\ell^2$ is not contained in \HE; generally not even via a bounded inclusion operator, it is true that each Dirac mass $\gd_x$ is a function of finite energy; see Definition~\ref{def:d_x} and Remark~\ref{rem:Diracs-in-HE}.)

\begin{defn}\label{def:J}
	Define $J$ to be the inclusion operator mapping $\spn\{\gd_x\}_{x \in G} \ci \ell^2(G)$ to $\spn\{\gd_x\}_{x \in G} \ci \HE$, so that
	\linenopax
	\begin{align}
		J:\ell^2(G) \to \HE \qq\text{by}\qq J:\gd_x \mapsto \gd_x. 
	\end{align}
	Note that $\spn\{\gd_x\}_{x \in G}$ is dense in $\ell^2(G)$ but not in \HE. In particular, $J$ is densely defined. Typically $J$ is unbounded, and so \emph{not} defined on all of $\ell^2(G)$. However, it is shown in Lemma~\ref{thm:closable} that $J$ is closable.
\end{defn}

In Definition~\ref{def:domLap2}, we discussed the Laplace operator $\Lap_{\ell^2}$ acting on $\ell^2(G)$ and in Definition~\ref{def:domLap2}, we discussed the Laplace operator $\Lap_{\HE}$ acting on \HE. In this section, we consider the same formal operator, but now as a transform from \HE to $\ell^2(G)$.

\begin{defn}\label{def:M}
	Define $M:\HE \to \ell^2(G)$ to be the operator with dense domain 
	\linenopax
	\begin{align}\label{eqn:domM}
		\dom M := \spn\{v_x\}_{x \in G}
	\end{align}
	by the pointwise formula
	\linenopax
	\begin{align}\label{eqn:M}
		(Mu)(x) := \la \gd_x, u\ra_{\HE}.
	\end{align}
\end{defn}

\begin{lemma}\label{lem:M}
	For any $x \in G$, one has the pointwise identity $Mu(x) = \Lap u(x)$. In particular, $Mv_x = \gd_x - \gd_o$.
\end{lemma}
\begin{proof}
	This is immediate from \eqref{eqn:v_x}.
\end{proof}

\begin{lemma}\label{thm:ScTandTcS}
	Suppose that for two Hilbert spaces \sH and \sK, we have operators $S:\sH \to \sK$ and $T:\sK \to \sH$ with dense domains $\dom S \ci \sH$ and $\dom T \ci \sK$. Then the following are equivalent:
	\linenopax
	\begin{align}
		(i) \;& \la S \gf,\gy\ra_\sK = \la \gf, T \gy\ra_\sH \text{ for all } \gf \in \dom S \text{ and }\gy \in \dom T. \label{eqn:ScTandTcS-1} \\
		(ii) \;& S \ci T^\ad \text{ and } T \ci S^\ad. \label{eqn:ScTandTcS-2}
	\end{align}
	If (i) or (ii) holds, then $S$ and $T$ are closable. 
	Moreover, $TS$ is essentially self-adjoint in \sH and $ST$ is essentially self-adjoint in \sK.
\end{lemma}

The proof of Lemma~\ref{thm:ScTandTcS} can be found in many classic references, e.g., \cite{DuSc88} or \cite{KadisonRingroseI}.

%
\begin{theorem}\label{thm:closable}
	For all $k \in \spn\{\gd_x\}_{x \in G}$ and for all $u \in \spn\{v_x\}_{x \in G}$, we have 
	\linenopax
	\begin{align}\label{eqn:M-iso}
		\la Jk,u\ra_{\HE} = \la k, Mu\ra_{\ell^2}.
	\end{align}
	In particular, $J$ and $M$ are closable.
\end{theorem}
\begin{proof}
	We only need to check \eqref{eqn:M-iso} for $k=\gd_y$ and $u=v_x$, and \eqref{eqn:v_x} gives
	\linenopax
	\begin{align}
		\la J \gd_y, v_x\ra_{\HE} 
		= \gd_y(x) - \gd_y(o)
		= \gd_x(y) - \gd_o(y)
		= \la \gd_y, \gd_x - \gd_o \ra_{\ell^2},
		= \la \gd_y, Mv_x \ra_{\ell^2},
	\end{align}
	where the last equality follows by Lemma~\ref{lem:M}. Now Lemma~\ref{thm:ScTandTcS} implies $J \ci M^\ad$ and $M \ci J^\ad$. Since $J$ has dense domain and $J \ci M^\ad$, it must also be the case that $M^\ad$ has dense domain. But $M^\ad$ has dense domain if and only if $M$ is closable. The same argument shows that $J$ is also closable.
\end{proof}

\begin{remark}\label{rem:clJ}
	Theorem~\ref{thm:closable} shows that $J$ is closable, so let us momentarily denote the closure of $J$ by $\cl J$ (but see Definition~\ref{def:J=clJ}). 
	Since $J \ci M^\ad$, it is clear that $\cl J \ci M^\ad$ as well. Our next result (Theorem~\ref{thm:Lbar=Mstar}) implies that in fact there is equality between the two operators (see Corollary~\ref{def:M=Lstar}). However, Theorem~\ref{thm:Lbar=Mstar} holds true for more general pairs of operators. For further background on closability, see \cite{MR0052042,MR600620}.
\end{remark}

\begin{theorem}\label{thm:Lbar=Mstar}
	Under the hypotheses of Theorem~\ref{thm:ScTandTcS}, assume further that $S \ci T^\ad$ and $T \ci S^\ad$. Then $TS$ is essentially self-adjoint if and only if $\cl{S} = T^\ad$ (where $\cl S$ is the closure of $S$).
\end{theorem}
\begin{proof}
	Let $\xG(S)$ denote the graph of the operator $S$, and let $V \ominus W$ denote the orthogonal complement of $W$ in the Hilbert space $V$. Since $S \ci T^\ad$ by \eqref{eqn:ScTandTcS-2}, the condition $\cl{S} = T^\ad$ means that $\xG(T^\ad) \ominus \xG(S) = 0$. Expressed in terms of $\sH \oplus \sK$, the latter is equivalent to 
	\linenopax
	\begin{align}
		\left\la 
		\left[\begin{array}{c} \gf_0 \\ T^\ad \gf_0 \end{array}\right], 
		\left[\begin{array}{c} \gy \\ T \gy \end{array}\right] \right\ra
		=0, \q\text{for all } \gy \in \dom T,
	\end{align}
	which implies that $\gf_0=0$. This, in turn, is equivalent to 
	\linenopax
	\begin{align}\label{eqn:}
		\la \gf_0, \gy\ra_\sH + \la T^\ad \gf_0, S\gy\ra_\sK = 0, \q \text{for all } \gy \in S,
	\end{align}
	i.e., that $T^\ad \gf_0 \in \dom S^\ad$ and $S^\ad T^\ad \gf_0 = -\gf_0$. In other words,
	\linenopax
	\begin{align}\label{eqn:defect-appearance}
		\gf_0 \in \dom (TS)^\ad
		\q\text{and}\q
		(TS)^\ad \gf_0 = -\gf_0.
	\end{align}
	Thus $\gf_0$ is a defect vector of the Hermitian operator $TS$ and \eqref{eqn:defect-appearance} has no nontrivial solutions if and only if $TS$ is essentially self-adjoint.
\end{proof}

We now apply the general result of Theorem~\ref{thm:Lbar=Mstar} to the specific operators at hand.

\begin{cor}\label{def:M=Lstar}
	For $J:\dom(J)\ci\ell^2(G) \to \HE$ and $M:\dom(M)\ci\HE\to\ell^2(G)$, we have $\cl J = M^\ad$ and $J^\ad = \cl M$.
\end{cor}
\begin{proof}
	Since $MJ = \Lap_{\ell^2}$ and $\Lap_{\ell^2}$ is essentially self-adjoint by \cite{SRAMO}, we have $\cl J = M^\ad$ immediately from Theorem~\ref{thm:Lbar=Mstar}, whence taking adjoints gives $J^\ad = M^{\ad\ad} = \cl M$.
\end{proof}

\begin{remark}\label{rem:funky-adjoint}
	Lemma~\ref{lem:M} and Corollary~\ref{def:M=Lstar} show that the adjoint of an inclusion operator can be a very different operator (in this case, the Laplacian). This may seem strange at first, but is simply an reflection of the different inner products in the two Hilbert spaces. For another example (in perhaps a more familiar context), see Example~\ref{exm:unboundedly-contained}.
\end{remark}

\begin{cor}\label{def:JJstar}
	$\cl J  J^\ad$ is a self-adjoint extension of \LapE.
\end{cor}

\begin{rem}\label{rem:JJ*-and-J*J}
	If $L=\cl J$, then since $L$ is closed, a theorem of von Neumann implies that both operators $LL^\ad$ and $L^\ad L$ are self-adjoint. In the next section, we will see in fact that $LL^\ad$ is the Krein extension of \LapE.
\end{rem}

\subsection{The Krein extension}

In Corollary~\ref{thm:ess-sa} (see also Remark~\ref{rem:domLap2}), we have shown that the graph Laplacian is essentially self-adjoint on its natural domain in $\ell^2(G)$, and in previous work \cite{SRAMO} we showed that the graph Laplacian may not be not essentially self-adjoint on its natural domain in \HE. Furthermore, we constructed the Friedrichs extension of the Laplacian on \HE in \cite{Friedrichs}. In this section, we take a look at the Krein extension.

\begin{defn}\label{def:Krein}
	 Let $S$ be a semibounded operator which is not self-adjoint, and suppose that $S^\ad$ has a nontrivial kernel. Let $T$ be a self-adjoint extension of $S$ which is defined on $\ker S^\ad$. If $\ker S^\ad \ci \ker T$, then $T$ is called the \emph{Krein extension} of $S$, and we denote it by $T=S^{(Kr)}$. By a theorem of Krein's, $T$ is the unique maximal self-adjoint extension of $S$, with respect to the usual ordering on self-adjoint extensions; see \cite{DuSc88} for details.
\end{defn}

\begin{defn}\label{def:J=clJ}
	Henceforth, to lighten the notation, we will abuse notation and denote the closure of $J$ again by $J$:
	\linenopax
	\begin{align}\label{eqn:J=clJ}
		J = \cl J.
	\end{align}
\end{defn}

\begin{lemma}\label{thm:J-adjoint}
	$J^\adj:\HE \to \ell^2(G)$ is defined by $J^\adj:v_x\mapsto \gd_x-\gd_o$.
\end{lemma}
\begin{proof}
	Since the reproducing kernel $\{v_x\}_{x \in G}$ is dense in \HE, it suffices to compute
	\linenopax
	\begin{align}
		\la J^\adj v_x, \gd_y\ra_{\ell^2} 
		&= \la v_x, J\gd_y\ra_{\HE} 
		= \la v_x, \gd_y\ra_{\HE} 
		= \gd_y(x)-\gd_y(o)
		= \gd_x(y) - \gd_o(y)
		= \la \gd_x - \gd_o, \gd_y\ra_{\ell^2}.
		\qedhere
	\end{align}
\end{proof}

\begin{lemma}\label{thm:kerJ*=Harm}
	$\ker J^\ad = \Harm$. 
\end{lemma}
\begin{proof}
	By definition, $u \in \dom J^\ad$ iff there is a $C<\iy$ for which
	\linenopax
	\begin{align}\label{eqn:J-ad-bound}
		\left|\la J \gf, u \ra_{\HE}\right| < C \|u\|_{\HE},
		\qq\text{for all } \gf \in \spn\{\gd_x\}_{x \in G}.
	\end{align}
	However, \eqref{eqn:dx-rep} implies $\la J \gd_x, u \ra_{\HE} = \la \gd_x, u \ra_{\HE} = \Lap u(x)$, so that \eqref{eqn:J-ad-bound} holds with $C=0$. This shows that all elements of \Harm are in $\dom J^\ad$, but the same calculation shows that $\ker J^\ad = \Harm$.
\end{proof}

\begin{theorem}\label{thm:JJ-Lap}
	(i) $J^\adj J = \Lap_{\ell^2}$, the self-adjoint Laplace operator on $\ell^2(G)$, and (ii) $JJ^\adj = \LapK$, the Krein extension of the Laplace operator on \HE.
\end{theorem}
\begin{proof}
	For the proof of (i), it suffices to work with the onb $\{\gd_x\}_{x \in G}$, in which case we use \eqref{eqn:d=v} to see the pointwise identity 
	\linenopax
	\begin{align}
		J^\adj J \gd_y
		= J^\adj \left(c(y)v_y - \sum_{z \nbr y} c_{yz}v_z\right)
		&= c(y)(\gd_y-\gd_o) - \sum_{z \nbr y} c_{yz}(\gd_z-\gd_o) \\
		&= c(y)\gd_y - \sum_{z \nbr y} c_{yz}\gd_z
		= \Lap_{\ell^2} \gd_y.
	\end{align}
	Since $S^\ad S$ is a self-adjoint operator for any closed operator $S$ with dense domain, the foregoing equation shows that $J^\ad J$ is a self-adjoint extension of $\Lap_{\ell^2}$. However, it was shown in \cite{SRAMO} that $\Lap_{\ell^2}$ is essentially self-adjoint, so the conclusion follows.
	
	For the proof of (ii), Lemma~\ref{thm:J-adjoint} gives
	\linenopax
	\begin{align}
		JJ^\ad v_x = J(\gd_x - \gd_o) = \gd_x - \gd_o = \Lap_{\HE} v_x,
	\end{align}
	so that $JJ^\ad$ is a self-adjoint extension of $\Lap_{\HE}$. Since $JJ^\ad$ is an extension of $\Lap_{\HE}$, the equality 
	\linenopax
	\begin{align}
		\ker \Lap_{\HE}^\ad = \ker \Lap_{\HE} = \Harm
	\end{align}
	shows $\ker \Lap_{\HE}^\ad \ci \ker JJ^\ad$, and the conclusion follows by Definition~\ref{def:Krein}.
\end{proof}

\begin{remark}\label{rem:Lenz-acknowledgement}
	The authors gratefully acknowledge the contributions of Daniel Lenz to the statement and proof of Theorem~\ref{thm:JJ-Lap}, which is crucial for the sequel.
\end{remark}

The following lemma is well-known; cf.~\cite{KadisonRingroseI} or \cite{DuSc88}, for example.

\begin{lemma}\label{thm:4-bounds}
	Let \sH, \sK be Hilbert spaces and suppose $J:\sH \to \sK$ is a closed operator with dense domain.
	Then
	\linenopax
		\begin{align*}
			\|J\|^2_{\sH \to \sK}
			= \|J^\ad\|^2_{\sK \to \sH}
			= \|J^\ad J\|_{\sH \to \sH}
			= \|JJ^\ad\|_{\sK \to \sK}.
		\end{align*}
	In particular, if one of them is finite, then they all are finite. 
\end{lemma}

\begin{remark}\label{rem:often-unbounded}
	In most interesting cases, the various incarnations of the Laplace operator on an infinite network are unbounded. However, Theorem~\ref{thm:equivalent-bounds} shows that in case any one of them is bounded, they are all bounded.
\end{remark}

\begin{theorem}\label{thm:equivalent-bounds}
	The following are equivalent:
	\begin{enumerate}
	\item $\Lap_{\ell^2} : \ell^2(G) \to \ell^2(G)$ is bounded.
	\item $\Lap_{\HE} : \HE \to \HE$ is bounded.
	\item $J : \ell^2(G) \to \HE$ is bounded.
	\end{enumerate}
	Moreover, in this case $\|\Lap_{\ell^2}\| = \|\Lap_{\HE}\| = \|J\|^2$.
\end{theorem}
\begin{proof}
	Observe that Theorem~\ref{thm:JJ-Lap} implies
	\linenopax
	\begin{align}
		\la \gx, \Lap_{\ell^2} \gx \ra_{\ell^2} 
		&= \la \gx, J^\ad J \gx \ra_{\ell^2} 
		= \la J\gx, J\gx \ra_{\HE}
		= \| J \gx \|_{\HE}^2
		\leq C \| \gx \|_{\HE}^2,
	\end{align}
	and similarly
	\linenopax
	\begin{align}
		\la u, \Lap_{\HE} u \ra_{\HE} 
		&= \la u, JJ^\ad u \ra_{\HE} 
		= \la J^\ad u, J^\ad u \ra_{\ell^2}
		= \| J^\ad u\|_{\HE}^2
		\leq C \| u \|_{\HE}^2.
	\end{align}
	Note that $J$ is closed (see \eqref{eqn:J=clJ}) and densely defined, so it is bounded if and only if $J^\ad$ is bounded.
	For the equality of the bounds for these operators, see Lemma~\ref{thm:4-bounds}.
\end{proof}

\begin{remark}\label{rem:sig-of-unbounded}
	The significance of Theorem~\ref{thm:equivalent-bounds} is that it shows that the inclusion mapping which sends $\gd_x \in \ell^2(G)$ to $\gd_x \in \HE$ produces an unbounded containment (in the sense of Definition~\ref{def:unboundedly-contained}) whenever \LapE is unbounded.
\end{remark}

\subsection{(Un)bounded containment}

\begin{defn}\label{def:unboundedly-contained}
	Let $K$ and $H$ be Hilbert spaces with $K \ci H$. We say that $K$ is \emph{boundedly contained} in $H$ iff the inclusion operator $J:K \hookrightarrow H$ is bounded, i.e., iff 
	\linenopax
	\begin{align}\label{eqn:inclusion-bound}
		\|k\|_H \leq C\|k\|_K, \qq \forall k \in K,
	\end{align}
	for some $C<\iy$.
	In terms of the operator norm, this means that $\|J\|_{K \to H} < \iy$. 
	
	We say that $K$ is \emph{unboundedly contained} in $H$ iff  \eqref{eqn:inclusion-bound} does not hold but $K$ has a dense linear subspace $K_0 \ci H$ for which the inclusion map $J:K_0 \hookrightarrow H$ is closable. 
\end{defn}

Note that $K$ is unboundedly contained in $H$ iff the domain of the adjoint operator $J^\ad:H \to K$ is dense in $H$, in which case
\linenopax
\begin{align}
	\la k_0, h \ra_H = \la k_0, J^\ad h \ra_K,
	\qq\text{for each } k_0 \in K_0 \text{ and } h \in \dom J^\ad.
\end{align}

If $K$ is unboundedly contained in $H$, then the graph 
\linenopax
\begin{align}
	\xG(J) 
	= \left\{\left[\begin{array}{c} \gx \\ J \gx \end{array}\right] \suth \gx \in \dom(J)\right\}
	\ci \left[\begin{array}{c} K \\ H \end{array}\right] = K \oplus H
\end{align}
is closed relative to the inner product
\linenopax
\begin{align}
	\left\la \left[\begin{array}{c} k_1 \\ h_1 \end{array}\right],
		\left[\begin{array}{c} k_2 \\ h_2 \end{array}\right]  \right\ra_{\xG(J)}
		:= \la k_1, k_2\ra_K + \la h_1, h_2\ra_H 
\end{align}

\begin{theorem}\label{thm:ell2-unboundedly-contained}
	$\ell^2(G)$ is unboundedly contained in \HE. 
\end{theorem}
\begin{proof}
	Consider the subspace $K_0 := \spn\{\gd_x \suth x \in G\}$, which is clearly dense in $K = \ell^2(G)$, and note that $K_0 \ci \HE$ by \eqref{eqn:d=v}. Theorem~\ref{thm:closable} shows that $J:\ell^2(G) \hookrightarrow \HE$ is closable and so the result follows.
\end{proof}

\begin{prop}[Paulsen, Thm.~5.1]\label{prop:RKHS-containment}
	If $H$ is a RKHS and $K$ is boundedly contained in $H$, then $K$ is also a RKHS. 
\end{prop}

If the inclusion is unbounded (as it is in the present case, by Theorem~\ref{thm:ell2-unboundedly-contained}), then it may not be the case that $K$ is also a RKHS: see Example~\ref{exm:unboundedly-contained}. However, in the present case, we are still able to construct an RKHS (see Theorem~\ref{thm:RKHS(Q)}). Note that $J\gd_x = \gd_x \in \HE$ (see Remark~\ref{rem:Diracs-in-HE}) and so we can define the positive definite function
\linenopax
\begin{align}
	Q(x,y) := \la J\gd_x, J\gd_y\ra_{\HE}.
\end{align}
Note from \eqref{eqn:dx-rep} that 
\linenopax
\begin{align}
	Q(x,y) = \Lap \gd_y(x) = 
	\begin{cases}
		c(x), &x=y,\\
		-c_{xy}, &x\nbr y\\
		0, \text{else},
	\end{cases}
\end{align}
so that the matrix for $Q$ is also the matrix for the $\iy \times \iy$ Laplace operator $\Lap_{\ell^2}$ on $\ell^2(G)$. It is known that $\Lap_{\ell^2}$ is a self-adjoint operator with dense domain in $\ell^2(G)$, and that it is semibounded (in fact, nonnegative):
\linenopax
\begin{align}
	\la \gx, \Lap_{\ell^2} \gx\ra_{\ell^2} \geq 0, 
	\qq\text{for all } \gx \in \dom \Lap_{\ell^2}.
\end{align}

\begin{exm}\label{exm:unboundedly-contained}
	Set $W = \{f \in L^2(\bR) \suth f' \in L^2(\bR)\}$, and 
	\linenopax
	\begin{align*}
		\|f\|_W^2 
		= \frac12\left(\int_\bR |f|^2 \dx + \int_\bR |f'|^2 \dx\right).
	\end{align*}
	Then $W$ is a RKHS with reproducing kernel $k(x,y) = e^{-|x-y|}$. To see this, note that $\frac{d^2}{dx^2} k_x = k_x-2\gd_x$, as distributions. Thus, integration by parts gives
	\linenopax
	\begin{align*}
		\la k_x, f\ra_W 
		= \frac12\left(\int_\bR k_x f \dy + \int_\bR k_x' f' \dy \right)
		= \frac12\left(\int_\bR k_x f \dy + \int_\bR (k_x-2\gd_x) f \dy \right)
		= f(x).
	\end{align*}
	Also, note that $L^2(\bR)$ is unboundedly contained in $W$, and $L^2(\bR)$ is \emph{not} a RKHS. It is clear that $L^2(\bR)$ is \emph{not} a RKHS because of examples like $f=|x|^{-1/2} \gx_{[-1,1]}$. To see the unbounded containment, let $J$ be the inclusion operator with domain $\dom J = C_c^\iy(\bR)$, the smooth functions of compact support, defined by
	\linenopax
	\begin{align*}
		J:\dom(J) \ci L^2(\bR) \to W
		\qq\text{by}\qq
		J\gf = \gf.
	\end{align*}
	Then $J^\ad f = \frac12(f-f'')$ and $\dom J^\ad = \{f \in W \suth f' \in W\}$.
\end{exm}

\begin{theorem}\label{thm:RKHS(Q)}
	The RKHS of $Q$ is the domain of $(\Lap_{\ell^2})^{1/2} = \sqrt{\Lap_{\ell^2}}$ (as defined by the spectral theorem).
\end{theorem}
\begin{proof}
 Let $\gx \in \ell^2(G)$ be of finite support, so that $\gx = \sum_{x \in F} \gx(x) \gd_x$, where $F$ is a finite subset of $G$. Then 
	\linenopax
	\begin{align}\label{eqn:Q-dom}
	\sum_{x,y \in G} \gx(x) \gx(y) Q(x,y)
	= \sum_{x \in G} \gx(x) (\Lap_{\ell^2} \gx)(x)
	= \la \gx, \Lap_{\ell^2} \gx\ra_{\ell^2}
	= \|\sqrt{\Lap_{\ell^2}}\gx\|_{\ell^2}^2.
	\end{align}
	Since $\sqrt{\Lap_{\ell^2}}$ is positive and self-adjoint, the completion on the right-hand side of \eqref{eqn:Q-dom} is the RKHS of $Q$.
\end{proof}

\begin{defn}
	Let $P_{\ell^2}$ denote the projection-valued measure of the self-adjoint operator $\Lap_{\ell^2}$, so that $f(\Lap_{\ell^2}) = \int_0^\iy \gl P_{\ell^2}(d\gl)$ for all Borel functions $f$ on $[0,\iy)$.
\end{defn}

\begin{cor}
	For all $\gx \in \dom \sqrt{\Lap_{\ell^2}}$, we have $\| \gx_Q\|_{\HE}^2 = \int_0^\iy \gl \| P_{\ell^2}(d\gl)\gx\|_{\ell^2}^2$.
\end{cor}
\begin{proof}
	This is immediate from \eqref{eqn:Q-dom}.
\end{proof}

\begin{lemma}[{\cite[Lem.~10.15]{OTERN}}]\label{thm:Parseval-frame}
	The system $\{\sqrt{c_{xy}} (v_x-v_y)\}_{(x,y) \in E}$ forms a Parseval frame for \HE, where $E$ denotes the set of (undirected) edges in $G$.
\end{lemma}
\begin{proof}
	Note that
	\linenopax
	\begin{align*}
		\frac12 \sum_{x,y \in G} \left|\la u, \sqrt{c_{xy}} (v_x-v_y) \ra_{\HE}\right|^2
		&= \frac12 \sum_{x,y \in G} \left|\sqrt{c_{xy}}(u(x)-u(y))\right|^2
		= \|u\|_{\HE}^2.
		\qedhere
	\end{align*}
\end{proof}

\subsection{Polar decomposition and spectral resolution}
In \cite{Friedrichs}, it is shown that the spectral measures of \LapF and $\Lap_{\ell^2}$ are mutually absolutely continuous with Radon-Nikodym derivative \gl (the spectral parameter); cf. \eqref{eqn:RN-deriv}. In this section, we show that \LapK and $\Lap_{\ell^2}$ have identical spectral representations.

\begin{lemma}\label{thm:U}
	There is an isometry $U:\ell^2(G) \to \HE$ such that
	\linenopax
	\begin{align}\label{eqn:U}
		J = U(J^\ad J)^{1/2} = (J J^\ad)^{1/2} U.
	\end{align}
\end{lemma}
\begin{proof}
	Existence of a partial isometry comes from the fact that $J$ is \emph{closed}. 
	Observe that $\ker J = \{0\}$ by Lemma~\ref{thm:no-harm-in-l2} and $\ker J^\ad = \Harm$ by Lemma~\ref{thm:kerJ*=Harm}. As a consequence, we have that $U$ is an isometry, instead of just a partial isometry.
\end{proof}

Now we have
	\linenopax
\begin{align*}
	U^\ad U = \id_{\ell^2}
	\q\text{and}\q UU^\ad = \proj_{\Fin},
\end{align*}
and Corollary~\ref{thm:spec(LapK)} follows by general theory (see, e.g., \cite{DuSc88,Rud91,ReedSimonI}). 

\begin{cor}\label{thm:spec(LapK)}
	$\spec_{\HE}(\LapK) \cup \{0\} = \spec_{\ell^2}(\Lap_{\ell^2}) \cup \{0\}$.
\end{cor}

\begin{defn}\label{def:P_Kr,P_ell2}
	Let $P_{Kr}$ and $P_{\ell^2}$ denote the projection-valued measures in the spectral resolutions of \LapK and $\Lap_{\ell^2}$, respectively.
\end{defn}

\begin{theorem}\label{thm:gmK=gm2}
	For all $\gx \in \ell^2(G)$, the spectral measures of \LapK and $\Lap_{\ell^2}$ are mutually absolutely continuous with Radon-Nikodym derivatives equal to 1:
	\linenopax
	\begin{align}\label{eqn:gmK=gm2}
		d\gm_{Kr}^{U\gx}(\gl) = d\gm_{\ell^2}^\gx(\gl). 
	\end{align}
	Furthermore, the restriction $\LapK|_{\Fin}$ is unitarily equivalent to $\Lap_{\ell^2}$ (recall from Theorem~\ref{thm:HE=Fin+Harm} that $\Fin = \HE \ominus \Harm = \cj{\spn\{\gd_x\}_{x \in G}}$).
\end{theorem}
\begin{proof}
	By the usual approximation procedure, Theorem~\ref{thm:JJ-Lap} combines with general spectral theory to give $U f(J^\adj J) = f(JJ^\adj)U$ first for when $f$ is a polynomial, and then for when $f$ is a Borel function. For $A \in \sB(\bR_+)$, we take a characteristic function $f = \gx_A$ and obtain 
	\linenopax
	\begin{align}\label{eqn:PKU=UP2}
		P_{Kr}U\gx = UP_{\ell^2}(A)\gx, \qq \text{for all } \gx \in \ell^2(G) \text{ and } A \in \sB(\bR_+).
	\end{align}
	For $f \in \HE$ and $\gx \in \ell^2(G)$, define
	\linenopax
	\begin{align}\label{eqn:muKU=Umu2}
		\gm_{Kr}^{f}(A) = \|P_{Kr}(A)f\|_{\HE}^2
		\q\text{and}\q
		\gm_{\ell^2}^\gx(A) = \|P_{\ell^2}(A)\gx\|_{\ell^2}^2,
		\qq\text{for all } A \in \sB(\bR_+).
	\end{align}
	Then, since $U$ is an isometry, we have
	\linenopax
	\begin{align*}
		\gm_{Kr}^{U\gx}(A) &= \|P_{Kr}(A)U\gx \|_{\HE}^2 = \|UP_{\ell^2}(A)\gx\|_{\ell^2}^2 = \|P_{\ell^2}(A)\gx\|_{\ell^2}^2 = \gm_{\ell^2}^\gx(A).
		\qedhere
	\end{align*}
\end{proof}

\begin{remark}
	For $P_{Kr}$, $P_{\ell^2}$ as in Definition~\ref{def:P_Kr,P_ell2}, we have the following.
	\begin{enumerate}
	\item $P_{\ell^2}(\{0\}) = 0$, since there are no nonzero harmonic functions in the domain of $\Lap_{\ell^2}$.
	\item $P_{Kr}(\{0\})h = h$ if and only if $h \in \HE$ and $\Lap h = 0$.
	\end{enumerate}
\end{remark}

%

\subsection{Comparison of the Krein and Friedrichs extensions}

We recall some definitions from \cite{Friedrichs}, in which the construction of the Friedrichs extension was carried out using a mapping \gF that sends $\gd_x \in \ell^2(G)$ to $w_x \in \HE$, in place of $J:\gd_x \to \gd_x$.

\begin{defn}\label{def:dipole}
  A \emph{monopole} is any $u \in \HE$ satisfying the pointwise identity $\Lap u = \gd_x$, for some vertex $x \in \Graph$. We denote a monopole solving this equation by $w_x$.
\end{defn}

  
\begin{remark}\label{rem:normalization}
  Denote the unique energy-minimizing monopole at $o$ by $w_o$; the existence of such an object is explained in \cite[\S3.1]{DGG}.
  We will be interested in the family of monopoles defined by
  \linenopax
  \begin{align}\label{eqn:w_x}
    \monov := w_o + v_x, \qq x \neq o,
  \end{align}
  and we use the representatives specified by
  \linenopax
  \begin{align}\label{eqn:w_x(o)}
    \monov(y) = \la \monov,\monoy\ra_{\HE} = \monoy(x), 
    \qq\text{and}\qq
    v_x(o)=0.
  \end{align}
  When $\Harm=0$, $\energy(\monov) = \la \monov,\monov\ra_{\HE} = \monov(x)$ is the \emph{capacity} of $x$; see, e.g., \cite[\S4.D]{Woess09}.
\end{remark}

\begin{defn}\label{def:Phi}
  Define $\gF : \ell^2(G) \to \HE$ on $\dom\gF = \spn\{\gd_x\}_{x \in G}$ by $\gF\gd_x = \monov$. 
\end{defn}

  Note that $\ran \gF$ is dense in \HE because it contains $\spn\{v_x\}_{x \in G}$; see \eqref{eqn:w_x}. 
  Now let $\gx \in \ell^2(G)$ so that \gx is a finite linear combination of Dirac masses. Then the respective spectral measures of $\LapF$ and $\Lap_{\ell^2}$ are
\linenopax
\begin{align}
	d\gm^{\gF\gx}_{\sF}(\gl) := \|P_\sF(d\gl)\gF\gx\|_\energy^2
	\qq\text{and}\qq
	d\gm^\gx_{\ell^2}(\gl) := \|P_{\ell^2}(d\gl)\gx\|_2^2.
\end{align}

\begin{theorem}[{\cite[Thm.~52]{Friedrichs}}]
	The spectral measures of \LapF and $\Lap_{\ell^2}$ are mutually absolutely continuous with Radon-Nikodym derivatives related by
\linenopax
\begin{align}\label{eqn:RN-deriv}
	\gl d\gm^{\gF\gx}_{\sF}(\gl) = d\gm^\gx_{\ell^2}(\gl).
\end{align}
Consequently, the spectral gaps of \LapF and $\Lap_{\ell^2}$ are identical.
\end{theorem}

\begin{cor}\label{thm:muKr=muF}
	The spectral measures of \LapF and \LapK are mutually absolutely continuous with Radon-Nikodym derivatives related by
\linenopax
\begin{align}\label{eqn:RN-deriv}
	\gl d\gm^{\gF\gx}_{\sF}(\gl) = d\gm^{U\gx}_{Kr}(\gl).
\end{align}
\end{cor}

\section{The Green operator}

In this section, we ignore domains for the moment and look at some formal computations. 

\begin{defn}
	Define the probability transition kernel by $p(x,y) = \frac{c_{xy}}{c(x)}$, and the probability transition operator by $Pf(x) = \sum_{y \nbr x} p(x,y) f(y)$. Note that we do not consider domains here. Writing $c$ for the multiplication operator, the (formal) \emph{Green operator} for \Lap is
	\linenopax
	\begin{align}\label{eqn:G}
		Gf := (\id - P)^{-1} \tfrac1c f.
	\end{align}
\end{defn}


\begin{theorem}\label{thm:G-defect}
	If the random walk on $G$ is aperiodic, there is a vector $f:G \to \bR$ satisfying $\Lap f = -f$ if and only if there is a harmonic function $h$ of the form $h = f+Gf$.
\end{theorem}
\begin{proof}
	Suppose that $h = f+Gf$. Then $\Lap h = \Lap f + f$, so $h$ is harmonic if and only if $\Lap f + f = 0$. Conversely, if $\Lap f = -f$, then $c(\id-P)f = f$, or $Pf = (1+\frac1c)f$. By iteration, one obtains $P^{n+1} f = f + \sum_{k=0}^n P^k(\frac fc)$. Letting $n$ tend to \iy gives
	\linenopax
	\begin{align}\label{eqn:harmonic-limit}
		\lim_{n \to \iy} P^n f = f + (\id-P)^{-1} \tfrac1c f,
	\end{align}
	and aperiodicity of the random walk (and irreducibility which follows from the connectedness of $G$) ensures that the limit on the right side exists. This implies that the left side is fixed by $P$, in which case it must be harmonic, and \eqref{eqn:harmonic-limit} reads $h = f + Gf$.
\end{proof}

\begin{remark}\label{rem:harmonic-energy}
	It is not clear in the above that $h \in \HE$ when one begins with $f \in \HE$. Clearly, one gets $h \in \HE$ if and only if $Gf \in \HE$, and so the existence of a defect vector can be expressed in terms of how $G$ acts on \HE. 
\end{remark}

\section{Examples}

\begin{exm}[Geometric integers]\label{exm:geometric-integers}
  For a fixed constant $c > 0$, let $(\bZ,c^n)$ denote the network with integers for vertices, and with geometrically increasing conductances defined by $\cond_{n-1,n} = c^{\max\{|n|,|n-1|\}}$ so that the network under consideration is
  \linenopax
  \begin{align*}
    \xymatrix{
      \dots \ar@{-}[r]^{c^3}
      & -2 \ar@{-}[r]^{c^2} 
      & -1 \ar@{-}[r]^{c} 
      & 0 \ar@{-}[r]^{c} 
      & 1 \ar@{-}[r]^{c^2} 
      & 2 \ar@{-}[r]^{c^3} 
      & 3 \ar@{-}[r]^{c^4} 
      & \dots
    }
  \end{align*} 
  as in \cite[Ex.~6.2]{DGG} and \cite[Ex.~3.18]{SC}, and fix $o=0$. 
  
  For $c = 1$, $\Harm=\{0\}$ and \LapE is bounded and Hermitian, and thus clearly self-adjoint. Indeed, using Fourier theory, one can show that $\HE \cong L^2\left((-\gp,\gp),\sin^2(\frac t2)\right)$ and under this transform, \LapE corresponds to the multiplication operator$M_\gf$, for $\gf(f) = 4 \sin^2 t$. See \cite[\S6.3]{Friedrichs}, for example. 
  
  However, for $c>1$, $\Harm = \spn\{h\}$, where
  \linenopax
  \begin{align}\label{eqn:geometric-harmonic}
		h(n) = \sgn(n) (1-c) \sum_{k=1}^{|n|} c^{-k} = \sgn(n) (1-c^{-|n|}),
	\end{align}
	as depicted in Figure~\ref{fig:geometric-harmonic}. 
	Note that $h \in \Harm$ because
	\linenopax
	\begin{align}\label{eqn:h-energy}
		\|h\|_{\HE}^2 
		= 2 \sum_{n=1}^\iy c^n((1-c^{-n}) - (1-c^{-(n-1)}))^2
		= 2 \sum_{n=1}^\iy c^n(c^{-n}(c-1))^2
 		= 2(c-1)^2 \sum_{n=1}^\iy c^{-n}
		= 2(c-1).
	\end{align}
	It is shown in \cite[\S4.2]{SRAMO} that \LapE is not self-adjoint and has deficiency indices $(1,1)$. 	Furthermore, for the positive geometric integers $(\bZ_+,c^n)$, a defect vector $\gf \in \HE$ is constructed which satisfies
  \linenopax
  \begin{align}\label{eqn:defect}
    \Lap \gf(n) = -\gf(n). 
  \end{align}
  This construction can be extended to provide a defect vector on $(\bZ,c^n)$ which is supported only on $\bZ_+$. Here, we provide a construction for a defect vector on $(\bZ,c^n)$ which vanishes only at $o$ and which is much simpler than the construction in \cite{SRAMO}. 
  
  Set $f(o)=0$ and let the reader provide the value of $f(1) \neq 0$. Then the defect equation \eqref{eqn:defect} determines the values of $f$ on the rest of \bZ. For $n >0$, \eqref{eqn:defect} becomes 
  \linenopax
  \begin{align*}
		f(n+1) - \frac{c+1}{c}\left(1 + \frac1{c^n(c+1)}\right)f(n) + \frac1cf(n-1) = 0,
	\end{align*}
	so the characteristic equation gives roots $x_1 = 1+o(c^{-n})$ and $x_2 = \frac12 + o(c^{-n})$, as $n \to \iy$. This implies $f(n+1)-f(n) = c^{-n}(f(1)-f(0)) + o(c^{-n})$, as $n \to \iy$, and it clear that a similar computation holds for $n<0$. Hence, we have
  \linenopax
  \begin{align*}
    \sum_{n=1}^\iy c^n \left(f(n+1)-f(n)\right)^2 
    &= \sum_{n=1}^\iy c^{-n} \left(c^n (f(n+1)-f(n))\right)^2 
     = \sum_{n=1}^\iy c^{-n} \left(f(1)-f(0) + o(1)\right)^2.
  \end{align*}
  This shows that $\|f\|_{\HE} \leq \frac{c\left(f(1)-f(0)\right)^2}{1-c} + O(1)$, and therefore $f \in \HE$.
  
	\begin{figure}
		\begin{centering}
		\scalebox{0.7}{\includegraphics{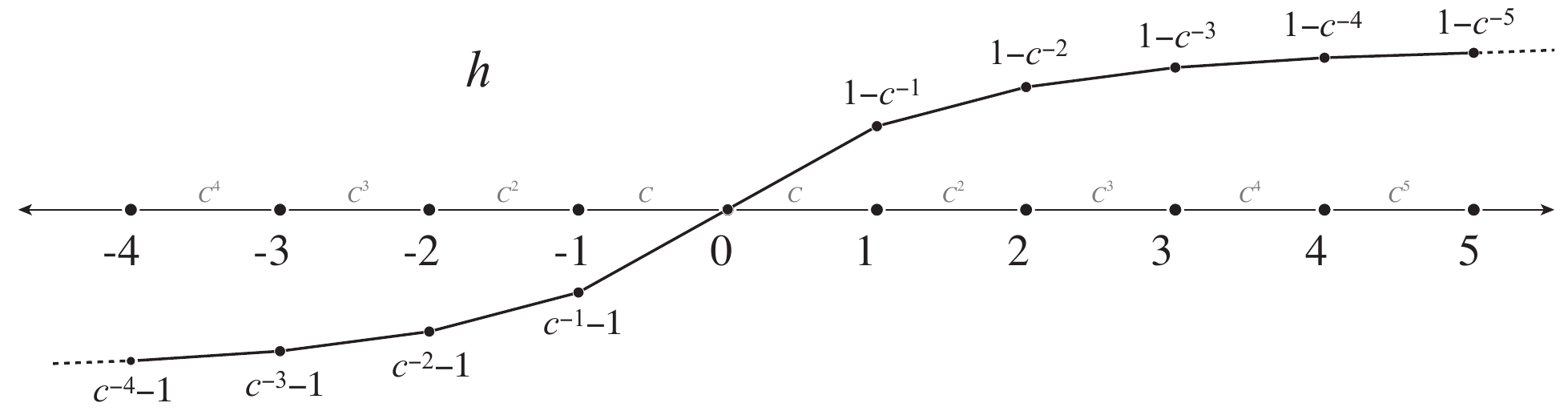}}
		\end{centering}
		\caption{\captionsize A harmonic function of finite energy on the geometric integers; see Example~\ref{exm:geometric-integers}.}
		\label{fig:geometric-harmonic}
	\end{figure}
\end{exm}

\begin{exm}[Geometric tree]\label{exm:geometric-tree}
	Consider the infinite binary tree \sT with root $o$ (the only vertex of degree 2) in which the conductances are given by 
  \linenopax
	\begin{align*}
		c_{xy} = c^{|x \wedge y|+1} 
		\qq\text{ for some constant $c \geq 1$}, 
	\end{align*}
	where $|x \wedge y|$ is the endpoint of the edge $(xy)$ that is closer to $o$. 
	
	For $c=1$, \Harm is separable with a basis that can be put in bijective correspondence with the Haar wavelets; see \cite{MR1918846, MR2511637, MR2680180} and the construction in \cite[Fig.~4]{ERM}. 
	Additionally, since this network is spherically symmetric in the sense of \cite[\S3.2]{Woj07}, it follows from \cite[Thm.~3.2.1]{Woj07} that \LapE is essentially self-adjoint, so the deficiency indices for \LapE are $(0,0)$.
	
	For $c>1$, one can specify any pair of rays from $o$ to \iy:
  \linenopax
	\begin{align*}
		R_+=(o=z^+_0,z^+_1,z^+_2,\dots)
		\q\text{and}\q
		R_-=(o=z^-_0,z^-_1,z^-_2,\dots),
	\end{align*}
	which satisfies the following conditions:
	\begin{itemize}
	\item $z^+_n \nbr z^+_{n-1}$ and $z^-_n \nbr z^-_{n-1}$, for all $n=1,2,\dots$.
	\item $|z^+_n|=|z^-_n|=n$, but $R_+$ and $R_-$ are disjoint except for $z^+_0=z^-_0=o$. 
	\end{itemize}
	Let $\gz:\bZ \to \sT$ be an embedding of the geometric integers into the tree via 
  \linenopax
	\begin{align*}
		\gz(n) = \begin{cases} z^+_n, &n \geq 0, \\ z^-_{-n}, &n\leq 0. \end{cases}
	\end{align*}
	Now one can define a harmonic function using \eqref{eqn:geometric-harmonic} as follows: 
  \linenopax
	\begin{align}\label{eqn:gh}
		\gh(x) = h(n), 
		\qq\text{where $\gz(n)$ is the nearest point of $\gz(\bZ)$ to $x$}.
	\end{align}
	Now \gh behaves like $h$ along the rays $R_+$ and $R_-$ in the sense that $\gh(z^+_n) = h(n)$ and $\gh(z^-_n) = h(-n)$.
	Also, \gh is locally constant on the complement of $\gz(\bZ)$. Letting $E(\sT)$ denote the edge set of \sT, this immediately gives 
  \linenopax
	\begin{align*}
		\|\gh\|_{\HE}^2 
		= \sum_{(xy) \in E(\sT)} c_{xy} (\gh(x)-\gh(y))^2
		= \sum_{(xy) \in \gz(\bZ)} c_{xy} (\gh(x)-\gh(y))^2
 		= \|h\|_{\HE}^2
	\end{align*}
because there is no contribution to the energy from edges not in $\gz(\bZ)$, by \eqref{eqn:gh}.
	
	This construction can be modified to produce a harmonic function $h_x$ based at any other vertex $x \in \sT$, by specifying  any pair of rays from $x$ to \iy:
  \linenopax
	\begin{align*}
		R_+=(x=z^+_0,z^+_1,z^+_2,\dots)
		\q\text{and}\q
		R_-=(x=z^-_0,z^-_1,z^-_2,\dots),
	\end{align*}
	satisfying the conditions
	\begin{itemize}
	\item $z^+_n \nbr z^+_{n-1}$ and $z^-_n \nbr z^-_{n-1}$, for all $n=1,2,\dots$.
	\item $|z^+_n|=|z^-_n|=|x|+n$, but $R_+$ and $R_-$ are disjoint except for $z^+_0=z^-_0=x$. 
	\end{itemize}
	Now define
  \linenopax
	\begin{align*}
		h^x(y) = \begin{cases} 1-c^{-n}, & y=z^+_n, \\ c^{-n}-1, &y=z^-_n, \\ h^x(y \wedge R), &y \notin R_+ \cup R_-, \end{cases}
	\end{align*}
	where $y \wedge R$ is the closest point of $R_+ \cup R_-$ to $y$.
	Thus, any isomorphic image of \bZ in \sT gives rise to a harmonic function of finite energy. Similarly, any isomorphic image of \bZ in \sT can be used to specify a defect vector of \LapE on \sT; see Figure~\ref{fig:geometric-tree}. Consequently, the deficiency indices for \LapE are $(\iy,\iy)$.

	\begin{figure}
		\begin{centering}
		\scalebox{0.9}{\includegraphics{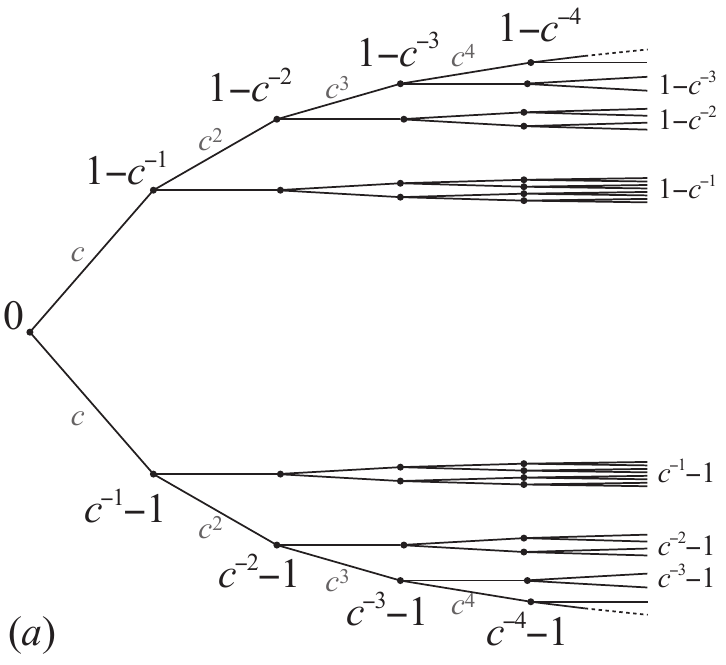}}
		\hstr[3]
		\scalebox{0.9}{\includegraphics{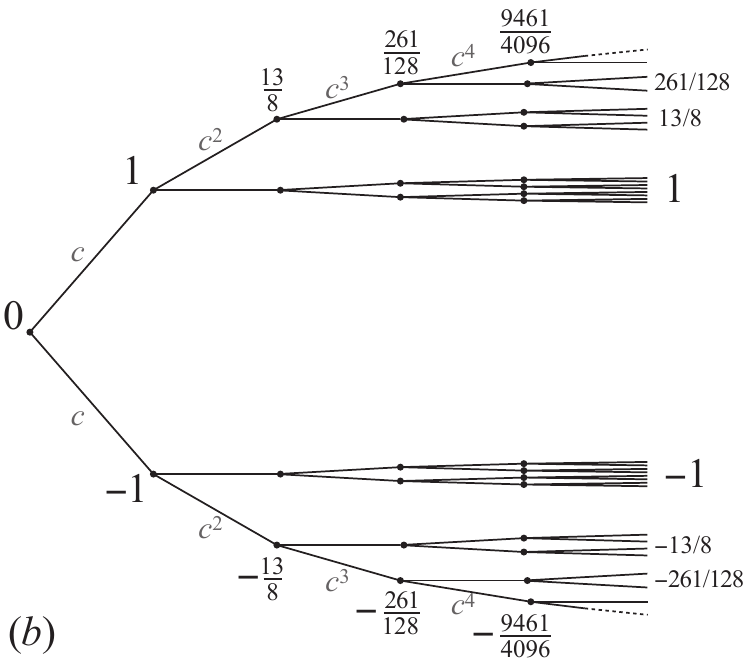}}
		\end{centering}
		\caption{\captionsize Function values are given in dark print; conductances in grey. (a) A harmonic function of finite energy on the geometric tree. (b) A defect vector on the geometric tree with $c=2$. See Example~\ref{exm:geometric-tree}.}
		\label{fig:geometric-tree}
	\end{figure}
\end{exm}

{\small
\bibliographystyle{math}
\bibliography{networks}
}

\vspace{2cm}

\end{document}